\documentclass[11pt]{amsart}
\usepackage{amsfonts}
\usepackage{ifthen, citesort}
\usepackage{amssymb}
\usepackage{amsrefs}
\usepackage{epic}

\DeclareMathOperator{\tr}{tr}

\newcommand{\rr}{\mathfrak{r}}
\def\R{\mathbb{R}}

\def\theta{\vartheta}
\def\phi{\varphi}
\def\epsilon{\varepsilon}
\newcommand{\A}[1]{\ifthenelse{#1 = 2}{\lvert A\rvert^{#1}}{\tr A^{#1}}}
\newcommand{\Akl}[2]{\ifthenelse{#1 = 2}%
{\left(\lvert A\rvert^{#1}\right)^{#2}}%
{\left(\tr A^{#1}\right)^{#2}}}

\newtheorem{theorem}{Theorem}[section]
\newtheorem{lemma}[theorem]{Lemma}

\newtheorem{corollary}[theorem]{Corollary}

\theoremstyle{definition}

\theoremstyle{remark}
\newtheorem{remark}[theorem]{Remark}

\numberwithin{equation}{section} 

\begin{document}

\title{Surfaces moving by powers of Gauss curvature}

\author{Ben Andrews}
\address{Mathematical Sciences Center, Tsinghua University;  Mathematical
Sciences Institute, Australia National University;  and Morningside Center
for Mathematics, Chinese Academy of Sciences.}
\email{Ben.Andrews@anu.edu.au}
\thanks{Research partly supported by Discovery Projects grants DP0985802 and DP120100097 of
 the Australian Research Council.  The authors are grateful for the hospitality of
 the Mathematical Sciences Centre of Tsinghua University where this work was carried out.}
\author{Xuzhong Chen}
\address{Mathematical Sciences Center, Tsinghua University}
 \email{cxzmath@yahoo.com.cn}

\dedicatory{}


\begin{abstract}
We prove that strictly convex surfaces moving by $K^{\alpha/2}$ become
spherical as they contract to points, provided $\alpha$ lies in the range $[1,2]$.  In the process
we provide a natural candidate for a curvature pinching quantity for surfaces
moving by arbitrary functions of curvature, by finding a quantity conserved by 
the reaction terms in the evolution of curvature. 
\end{abstract}

\maketitle

\section{INTRODUCTION}\label{introduction sec}

In this paper we study the contraction of smooth, compact, convex surfaces
with speed given by a power of the Gauss curvature:  That is, we
consider a family of embeddings $X:\ S^2\times[0,T)\to\R^3$ such that
\begin{equation}\label{flow eqn}
\frac{\partial X}{\partial t}=-K^{\alpha/2}\nu
\end{equation}
where $\alpha \in [1 , 2]$, $K$ is the Gauss curvature and
$\nu$ denotes the outer unit normal to the evolving surface $M_t=X(S^2,t)$.\par
 For $\alpha=2$ this flow coincides with the Gauss curvature
 flow, which was introduced by Firey \cite{Firey} as a model for
 the shape of wearing stones.
 Firey conjectured that the surfaces should become spherical as
 they contract to points in this process,
 and this was confirmed by the first author in \cite{AndrewsStones}.
 Bennett Chow \cite{Chow1} showed that flow
 by $K^{\alpha/n}$
 with $\alpha>0$ shrinks convex hypersurfaces in
 $\R^{n+1}$ to points in finite
 time, and that the limiting shape is spherical if $\alpha\geq 1$ and
  the curvature of the initial hypersurface is sufficiently pinched.
\par
  Our main theorem is the following:
\begin{theorem}\label{main thm}
Let $M_{0}=X_{0}(M)$ be a compact, smooth, strictly convex surface
in $\R^3$, given by an embedding $X_0$. Then there exists a maximal $T>0$ and a unique,
smooth solution $\{M_t=X_t (M)\}$ of the evolution
equation (\ref{flow eqn}) with $M_{0}=X_{0}(M)$. The surfaces $M_t$
are smooth and strictly convex, and converge to $q\in \R^3$ as $t$ approaches
$T$. Rescaling about $q$ gives smooth convergence to a sphere
\begin{equation}
\tilde{X}(t)=\frac {X-q} {\left((\alpha+1)(T-t)\right)^{\frac {1}
{\alpha +1}}}\rightarrow \tilde{X}_{T}
\end{equation}
in $C^{\infty}$, where $\tilde{X}_{T}$ is a smooth embedding with
$\tilde{X}_{T}(M)=\mathbb{S}^{2}(1)\subset \R^3$.
\end{theorem}
The existence and uniqueness, smoothness and convexity of the
solutions, and the uniform convergence to a point, were all proved in
\cite{Chow1}. Our contribution is the last part of the above theorem.
The key step in the proof is a pinching estimate for the principal curvatures,
 analogous to that used in  \cite{AndrewsStones}.  We remark that
 the evolution of surfaces by curvature
 is rather well understood in the case where the speed is homogeneous
 of degree one in the principal curvatures
  \cite{AndrewsNonconcave}, but far less is known for other degrees
  of homogeneity:  There are a few flows where
 curvature estimates have been found which are strong enough to deduce
 that limiting shapes are
spheres, starting with the Gauss curvature flow in
\cite{AndrewsStones}. A collection of other flows have been treated
\cites{Schn,Schu}, including certain powers of mean curvature and
certain sums of powers of principal curvatures, for which curvature
pinching estimates were found using a computational search
algorithm.  However at present we are far from understanding such flows of surfaces in any
generality.  One of the contributions of the present paper is a natural
candidate for a quantity which measures the pinching of principal
curvatures to each other as the surface shrinks to a point:  We solve the
system of ordinary differential equations associated with the `reaction' terms 
in the evolution of curvature for quite general flows by functions of curvature.
We hope this method may be useful for a wider class of flows, and mention 
at the end of the paper several other classes of flows for which this quantity is
preserved. \par

\section{PROOF OF THE THEOREM}\label{preliminary sec}
We work with a description of uniformly convex surfaces using their
support function as in \cite{AndrewsHarnack}: For each $z\in
S^2\subset\R^3$ we define $s(z,t)=\sup\{x\cdot z:\ x\in M_t\}$. Then
the surface $M_t$ is recovered by the embedding $X(z,t) =
s(z,t)z+\nabla s(z,t)$, where $\nabla s$ is the gradient vector of
$s(.,t)$ on $S^2$, interpreted as a vector in $\R^3$ tangent to
$S^2$ at $z$.  The principal radii of curvature (denoted $r_1$ and
$r_2$) are the eigenvalues of the symmetric positive definite
bilinear form $\rr_{ij}=\nabla_i\nabla_j s+g_{ij}s$, where $g_{ij}$
is the standard metric on $S^2$, and $\nabla$ the corresponding
connection.  The evolution equation \eqref{flow eqn} is then
 equivalent to the scalar parabolic equation
\begin{equation}\label{flow S}
\frac{\partial s}{\partial t} = F\left(\rr_{ij}\right),
\end{equation}
where $F(A) = -\det(A)^{-\alpha/2}$.  The Codazzi identity implies that
 $\nabla_i\rr_{jk}$ is totally symmetric.  The following evolution
 equation holds for $\rr_{ij}$ (see for example \cite{AndrewsMHGC}*{Equation 20}):
\begin{equation}\label{eq:evolve-r}
\frac{\partial}{\partial t}\rr_{ij} = \dot
F^{kl}\nabla_k\nabla_l\rr_{ij}+ \ddot
F^{kl,mn}\nabla_i\rr_{kl}\nabla_j\rr_{mn} +(F+\dot F^{kl}\rr_{kl})g_{ij}-\dot
F^{kl}g_{kl}\rr_{ij},
\end{equation}
where $\dot F^{ij}=\frac{\alpha}{2}\det
(A)^{-\alpha/2}(\rr^{-1})^{ij}$ is the matrix of derivatives of $F$
with respect to the components of $\rr_{ij}$, and $\ddot F$ gives
the corresponding second derivatives.

\subsection{The pinching estimate}\label{pinching sec}
In this section we prove an pinching estimate of curvature and
convergence of the evolving surfaces. We begin with the following
lemma.

\begin{lemma}\label{A2 conv pre}
If $\rr_{ij}\leq Cg_{ij}$ at $t=0$, then this remains true for all $t\in [0,T)$.
\end{lemma}

\begin{proof}
Consider $M_{ij}=\rr_{ij}-Cg_{ij}$ with $C>0$ so large that
$M_{ij}\leq 0$
 at $t=0$. We wish to show that $M_{ij}\leq 0$ for $t\in
 [0,T)$. By \eqref{eq:evolve-r}, we have:
\begin{align*}
\frac{\partial}{\partial t}M_{ij} = \dot
F^{kl}\nabla_k\nabla_lM_{ij}+ \ddot
F^{kl,mn}\nabla_iM_{kl}\nabla_jM_{mn} +(1-\alpha)Fg_{ij}-\dot
F^{kl}g_{kl}\rr_{ij},
\end{align*}
Let $v$ be a zero eigenvector of $M_{ij}$ with $|v|=1$,
$M_{ij}v^{j}=(\rr_{ij}-Cg_{ij})v^{j}=0$. Then we have
\begin{align*}
\left((1-\alpha)Fg_{ij}-\dot F^{kl}g_{kl}\rr_{ij}\right)v^iv^j=F
\left((1-\alpha)g_{ij}+\alpha g_{ij}\right)v^iv^j =F<0
\end{align*}
since $F$ is a concave function of the components $\rr_{ij}$.
The result follows from Hamilton's maximum principle
\cite{Ham3D}*{Theorem 9.1}.
\end{proof}

\begin{theorem}\label{curvature pinch}
Let $\{M_t=X(M,t)_{0\leq t<T}\}$ be a smooth, strictly convex
solution of the flow equation (\ref{flow eqn}). Then
\begin{equation}\label{curvature bound}
\sup_{M}\frac{(r_1(x,t)-r_2(x,t))^2} {(r_1(x,t)r_2(x,t))^{\alpha}}
\leq\sup_{M}\frac{(r_1(x,0)-r_2(x,0))^2}
{(r_1(x,0)r_2(x,0))^{\alpha}}
\end{equation}
\end{theorem}

\begin{proof}
We order the principal radii so that $r_2>r_1$. Then $r_{2}$ and
$r_{1}$ may not be smooth functions on $M$, but they are smooth on a
neighbourhood of any non-umbillic point.  For small $\varepsilon>0$
we will prove the negativity of the quantity
\begin{align*}
G=-F(r_2-r_1)-C_t
\end{align*}
where $C_t=C_0+\varepsilon(1+t)$ and $C_0=\sup_{M}\frac{r_2(x,0)-r_1(x,0)}
{(r_1(x,0)r_2(x,0))^{\alpha/2}}$.  This choice makes $G$ strictly negative for $t=0$.
 Note that $G$ is smooth near any point where $G=0$, since such a point cannot be umbillic.
  If $G$ is not everywhere negative, we consider the first time $t$ where the maximum of $G$
   reaches zero, and a point $p$ where this occurs.  Using \eqref{eq:evolve-r} we obtain
the following evolution equation for $G$:
\begin{align}\label{eq:evolve-G}
\begin{split}
\frac{\partial G}{\partial
t}=&\dot{F}^{ij}\nabla_i\nabla_jG+(\dot{G}^{ij}
\ddot{F}^{kl,mn}-\dot{F}^{ij}\ddot{G}^{kl,mn})\nabla_{i}\rr_{kl}\nabla_{j}\rr_{mn},\\
+&\left(F+\dot F^{ij}\rr_{ij}\right)\dot{G}^{ij}g_{ij}-\dot
F^{kl}g_{kl}\dot{G}^{ij}\rr_{ij}-\varepsilon.
\end{split}
\end{align}
Choose local coordinates for $M$ near
$p$ such that  $\rr_{ij}(p,t) = \textrm{diag}(\lambda_1,\lambda_2)$.
For convenience we use the notation $\dot f^i=\frac{\partial F}{\partial r_i}$
and $\ddot f^{ij}=\frac{\partial^2 F}{\partial r_i\partial r_j}$,
and similarly $\dot g^i=\frac{\partial G}{\partial r_i}$ and $\ddot g^{ij}=
\frac{\partial^2 G}{\partial r_i\partial r_j}$.
We now estimate the second term on the right hand side of the above
evolution equation.   To do this we
use the method of \cite{AndrewsICM} to reduce the negativity of this term
to a pair of inequalities.
The computation uses
results from \cite{AndrewsPinching} which give the first and second derivatives of
$F$ and $G$ in the above orthonormal frame:
\begin{align*}
\ddot{F}^{11,11}=& \ddot f^{11},\quad \ddot{F}^{22,22}=\ddot f^{22},
\quad \ddot F^{11,22}=\ddot f^{12},\\
\ddot{F}^{12,12}&=\ddot{F}^{21,21}=\frac{\dot f^1-\dot f^2}{r_1-r_2}
\end{align*}
It follows
that the second term on the right hand side of \eqref{eq:evolve-G} is as follows:
\begin{align*}
Q:=&\left(\dot{G}^{ij}\ddot{F}^{kl,mn}-\dot{F}^{ij
}\ddot{G}^{kl,mn}\right)\nabla_i\rr_{kl}\nabla_j\rr_{mn}\\
=&\left(\dot g^1\ddot f^{11}-\dot f^1\ddot g^{11}\right)(\nabla_1\rr_{11})^2
+\left(\dot g^1\ddot f^{22}-\dot f^1\ddot g^{22}\right)(\nabla_1\rr_{22})^2\\
+&2\left(\dot g^1\ddot f^{12}-\dot f^1\ddot g^{12}\right)\nabla_1\rr_{11}\nabla_1\rr_{22}
+2\frac{\dot g^1\dot f^2-\dot f^1\dot g^2}{r_2-r_1}\left(\nabla_1\rr_{12}\right)^2\\
+&\left(\dot g^2\ddot f^{11}-\dot f^2\ddot g^{11}\right)(\nabla_2\rr_{11})^2
+\left(\dot g^2\ddot f^{22}-\dot f^2\ddot g^{22}\right)(\nabla_2\rr_{22})^2\\
+&2\left(\dot g^2\ddot f^{12}-\dot f^2\ddot g^{12}\right)\nabla_2\rr_{11}\nabla_2\rr_{22}
+2\frac{\dot g^1\dot f^2-\dot g^2\dot f^1}{r_2-r_1}(\nabla_2\rr_{12})^2
\end{align*}
 At the maximum
point $(p, t)$ of $G$, we have $G=0$ and $\nabla_iG=0$. The latter gives
two equations which may be written as follows:
\begin{equation}\label{eq:1stderiv}
T_1:=\frac{\nabla_1\rr_{11}}{\dot g^2}=-\frac{\nabla_1\rr_{22}}{\dot g^1};\quad
T_2:=\frac{\nabla_2\rr_{22}}{\dot g^1}=-\frac{\nabla_2\rr_{11}}{\dot g^2}.
\end{equation}
These identities and the Codazzi symmetries $\nabla_1\rr_{12} =
\nabla_2\rr_{11}$ and $\nabla_2\rr_{12} = \nabla_1\rr_{22}$  reduce $Q$ to a linear combination of
$(\nabla_1\rr_{22})^2$ and $(\nabla_2\rr_{11})^2$:
These identities and the Codazzi symmetries $\nabla_1\rr_{12} =
\nabla_2\rr_{11}$ and $\nabla_2\rr_{12} = \nabla_1\rr_{22}$ reduce $Q$ to a linear combination of
$T^2_1$ and $T_2^2$:  We have $Q=Q_1T_1^2+Q_2T_2^2$, where
\begin{align*}
Q_1&=\left(\dot g^1\ddot f-\dot f^1\ddot g\right)(\dot g^2e_1-\dot g^1e_2,\dot g^2e_1-\dot g^1e_2)
+2(\dot g^1)^2\frac{\dot g^1\dot f^2-\dot g^2\dot f^1}{r_2-r_1};\\
Q_2&=\left(\dot g^2\ddot f-\dot f^2\ddot g\right)(\dot g^2e_1-\dot g^1e_2,\dot g^2e_1-\dot g^1e_2)
+2(\dot g^2)^2\frac{\dot g^1\dot f^2-\dot g^2\dot f^1}{r_2-r_1}.
\end{align*}
Now using the expression $g=-f(r_2-r_1)-C_t$ we obtain
\begin{equation*}
\begin{aligned}
\dot g^1 &= f-\dot f^1(r_2-r_1);\\
\dot g^2 &= -f-\dot f^2(r_2-r_1);
\end{aligned}\quad\text{\rm and}\quad
\begin{aligned}
\ddot g^{11}&=2\dot f^1-\ddot f^{11}(r_2-r_1);\\
\ddot g^{12}&=\dot f^2-\dot f^1-\ddot f^{12}(r_2-r_1);\\
\ddot g^{22}&=-2\dot f^2 -\ddot f^{22}(r_2-r_1).
\end{aligned}
\end{equation*}
The terms involving $\ddot f$ produce significant cancellations, resulting in the following:
\begin{align}
Q_1 &= f\ddot f(\dot g^2e_1-\dot g^1e_2,\dot g^2e_1-\dot g^1e_2)\label{eq:Q1}\\
&\quad\null+2f(\dot f^1+\dot f^2)\left(\frac{f^2}{r_2-r_1}-2f\dot f^1-\dot f^1\dot f^2(r_2-r_1)\right);\notag\\
Q_2 &= -f\ddot f(\dot g^2e_1-\dot g^1e_2,\dot g^2e_1-\dot g^1e_2)\label{eq:Q2}\\
&\quad\null+2f(\dot f^1+\dot f^2)\left(\frac{f^2}{r_2-r_1}+2f\dot f^2-\dot f^1\dot f^2(r_2-r_1)\right).\notag
\end{align}
Now we use the homogeneity of the speed $f$:  Write $f=-k^{-\alpha}$, where $k$ is homogeneous of degree one.  This gives
\begin{equation*}
\dot f:=\alpha k^{-(1+\alpha)}\dot k;\quad\text{\rm and}\quad
\ddot f:=-\alpha(1+\alpha)k^{-(2+\alpha)}\dot k\otimes \dot k+\alpha k^{-(1+\alpha)}\ddot k.
\end{equation*}
The term  $\ddot f(\dot g^2e_1-\dot g^1e_2,\dot g^2e_1-\dot g^1e_2)$ in \eqref{eq:Q1}
and \eqref{eq:Q2} appears to involve fourth powers of $\alpha$, with two powers from $\ddot k$
and one from each factor $\dot g^2e_1-\dot g^1e_2$.  However,  these are proportional
 to $(\dot k\otimes\dot k)(\dot k^2e_1-\dot k^1e_2,\dot k^2e_1-\dot k^1e_2)
 =(\dot k(\dot k^2e_1-\dot k^1e_2)^2$, which is zero.  Thus this term involves
  at most three powers of $\alpha$.  Since there is a common factor $\alpha$,
  this term is equal to $\alpha$ multiplied by a quadratic function of $\alpha$.
   The other terms in $Q_1$ and $Q_2$ are of degree at most three in $\alpha$ and
   have a factor $\alpha$, so we have the following:

\begin{lemma}\label{lem:quadratic}
$\frac{Q_1}{\alpha}$ and $\frac{Q_2}{\alpha}$ are quadratic functions of $\alpha$.
\end{lemma}

We now show the following:

\begin{lemma}\label{lem:convex}
$\frac{Q_1}{\alpha }$ and $\frac{Q_2}{\alpha}$ are convex functions of $\alpha$.
\end{lemma}

\begin{proof}
By Lemma \ref{lem:quadratic} we must show the $\alpha^3$ terms in
$Q_1$ and $Q_2$ are non-negative.   In the first line of
\eqref{eq:Q1},  $-\alpha^3k^{-3-4\alpha} (r_2-r_1)^2\ddot k(\dot
k^2e_1-\dot k^1e_2,\dot k^2e_1-\dot k^1e_2)$ is the only such term.
 From the second line of \eqref{eq:Q1} only the last term in the bracket contributes, yielding
$2\alpha^3k^{-3-4\alpha}(r_2-r_1)\dot k^1\dot k^2(\dot k^1+\dot k^2)$.  In the
expression \eqref{eq:Q2} for $Q_2$ the same two terms arise, but the sign of the
first is reversed.  Thus the condition for both $Q_2/\alpha$ and $Q_1/\alpha$ to be convex is precisely
$$
2\dot k^1\dot k^2(\dot k^1+\dot k^2)\geq (r_2-r_1)\left|\ddot k(\dot
k^2e_1-\dot k^1e_2,\dot k^2e_1-\dot k^1e_2)\right|.
$$
For Gauss curvature flows $k=\sqrt{r_1r_2}$, so the left hand side is
 $\frac14\left(\sqrt{\frac{r_2}{r_1}}+\sqrt{\frac{r_1}{r_2}}\right)$,
while the right is
$\frac14\left|\sqrt{\frac{r_2}{r_1}}-\sqrt{\frac{r_1}{r_2}}\right|$.
 The inequality clearly holds.
\end{proof}

It was proved in \cite{AndrewsStones} that $Q_1$ and $Q_2$
are non-positive when $\alpha=2$,
and in \cite{AndrewsNonconcave} that $Q_1$ and $Q_2$ are
non-positive for $\alpha=1$ (for any $k$).
 By Lemma \ref{lem:convex}, $Q_1$ and $Q_2$ are non-positive
 for any $\alpha\in[1,2]$, and hence $Q\leq 0$.

Now we consider the first two terms on the second line of \eqref{eq:evolve-G}:
\begin{align*}
Z:&=\left(F+\dot F^{ij}\rr_{ij}\right)\dot{G}^{ij}g_{ij}-\dot
F^{kl}g_{kl}\dot{G}^{ij}\rr_{ij}\\
&=(f+\dot f^1r_1+\dot f^2r_2)(\dot g^1+\dot g^2)-(\dot f^1+\dot f^2)(\dot g^1r_1+\dot g^2r_2)\\
&=(f+\dot f^1r_1+\dot f^2r_2)(-\dot f^1(r_2-r_1)+f-\dot f^2(r_2-r_1)-f)\\
&\quad\null-(\dot f^1+\dot f^2)((-\dot f^1(r_2-r_1)+f)r_1+(-\dot f^2(r_2-r_1)-f)r_2)\\
&=-(f+\dot f^1r_1+\dot f^2r_2)(\dot f^1+\dot f^2)(r_2-r_1)\\
&\quad\null -(\dot f^1+\dot f^2)(-f-\dot f^1r_1-\dot f^2r_2)(r_2-r_1)\\
&=0.
\end{align*}
Since the last term on the second line of \eqref{eq:evolve-G} is strictly negative,
we arrive at a contradiction.  Therefore $G$ remains negative as long as the solution
exists, and allowing $\varepsilon$ to approach zero proves the Theorem.
\end{proof}
\begin{corollary}\label{pinching cor}
For a smooth compact strictly convex surface $M_t$ in $\R^3$,
flowing according to $\frac{\partial X}{\partial t}=-K^{\alpha/2}\nu$ with $1\leq \alpha\leq 2$,
there exists $C_1=C_1(M_0,\alpha)$ such that
$0<\frac{1}{C_1}\le\frac{r_2}{r_1}\le C_1$.
\end{corollary}
\begin{proof}
Choose $C>0$ such that $r_1,\,r_2<C$ at $t=0$. Lemma \ref{A2 conv
pre} and Theorem \ref{curvature pinch} implies that
$$C^{2-2\alpha}\left(\frac{r_1}{r_2}+\frac{r_2}{r_1}-2\right)
\leq (r_1r_2)^{1-\alpha}\frac{(r_1-r_2)^2}{r_1r_2}
=\frac{(r_1-r_2)^2}{(r_1r_2)^{\alpha}}\le C_0^2.$$ We obtain the
bound on $\frac{r_2}{r_1}$ claimed above.
\end{proof}

\subsection{Convergence}
Given the pinching estimate of Theorem \ref{curvature pinch}, the
proof of convergence is rather straightforward:  We already know
that the surfaces remain smooth until they contract to a point, from
the result of \cite{Chow1}.  From Theorem \ref{curvature pinch} we
can deduce a strong pinching result of the kind used in
\cite{AndrewsMcCoy} as follows: Ordering the principal radii as
$r_2\geq r_1$ as before, the principal curvatures $\kappa_i=1/r_i$
satisfy $\kappa_1\geq \kappa_2$, and Theorem \ref{curvature pinch}
implies that
$$
\kappa_1-\kappa_2\leq C_0(\kappa_1\kappa_2)^{\frac{2-\alpha}{2}}.
$$
By corollary \ref{pinching cor} we have
$$
\kappa_1-\kappa_2\leq C_0C_1^{\frac{2-\alpha}{2}}\kappa_2^{2-\alpha}.
$$
For $\alpha\in (1,2)$ we have $2-\alpha\in(0,1)$. By Young's inequality we have for any $\varepsilon>0$
$$
\kappa_2^{2-\alpha}\leq \varepsilon \kappa_2+ (\alpha-1)(2-\alpha)^{\frac{2-\alpha}{\alpha-1}}\varepsilon^{-\frac{2-\alpha}{1-\alpha}}.
$$
Thus it follows that for any $\varepsilon>0$ there exists $C(\varepsilon)$ such that
$$
\kappa_1\leq (1+\varepsilon)\kappa_2+C(\varepsilon).
$$
The argument of  \cite{AndrewsMcCoy}*{Theorem 3.1} applies to prove
that for any $\delta>0$ there exists $r(\delta)>0$ such that
 the ratio of circumradius $r_+(M_t)$ to inradius $r_-(M_t)$ is bounded by $1+\delta$ provided $r_+(M)<r(\delta)$.
 Then \cite{AndrewsMcCoy}*{Proposition 12.1} applies to prove that the speed $K^{\alpha/2}$ is bounded
 below by $C(T-t)^{-\frac{\alpha}{1+\alpha}}$ for $t$ sufficiently close to $T$, and the higher regularity
  and convergence of the rescaled hypersurfaces follow as described in \cite{AndrewsMcCoy}*{Section 13}.

\section{REMARKS AND EXTENSIONS}
\begin{remark}
In the proof of the pinching estimate we used the structure of the Gauss curvature flows only in a few places:
In the proof that the zero order terms $Z$ in equation \eqref{eq:evolve-G} vanish we did not use any information
 about the speed, so this part of the argument works for arbitrary flows of surfaces by curvature, and amounts
 to a closed form solution of the `reaction' part of the evolution of curvature in such flows.  In the proof
  of non-positivity of the gradient terms $Q$ we used the homogeneity of the speed, but only the particular
  speed $K^{\alpha/2}$ at the last stage.  We expect this argument will apply for a reasonably large family
  of flows with high degree of homogeneity $\alpha$, but emphasise that the argument  relied on the fact
  that the case $\alpha=2$ is already known from \cite{AndrewsStones} (the case $\alpha=1$ is known for
   very general flows from \cite{AndrewsNonconcave}).
\end{remark}

\begin{remark}
Although we used only the highest order terms in $\alpha$ in the proof, one can of course compute
$Q_1$ and $Q_2$ completely for flows by powers of Gauss curvature.  This is a somewhat lengthy but
 straightforward computation, and yields
\begin{align*}
\frac{Q_1}{2\alpha}&\!=\!\frac{(2\alpha^2\!-\!5\alpha\!+\!2) r_1^2
r_2^3\!-\!(4\alpha^2\!-\!7\alpha\!+\!6)r_1^3r_2^2
\!+\!(2\alpha^2\!-\!3\alpha\!-\!2)r_1^4r_2\!+\!(\alpha\!-\!2)r_1^5}{(r_1r_2)^{\alpha/2+2}(r_2-r_1)(\alpha
r_2+(2-\alpha)r_1)^2}\\
\frac{Q_2}{2\alpha}&\!=\!
\frac{(2\alpha^2\!-\!5\alpha\!+\!2) r_2^2
r_1^3\!-\!(4\alpha^2\!-\!7\alpha\!+\!6)r_2^3r_1^2
\!+\!(2\alpha^2\!-\!3\alpha\!-\!2)r_2^4r_1\!+\!(\alpha\!-\!2)r_2^5}{(r_1r_2)^{\alpha/2+2}(r_2-r_1)(\alpha
r_1+(2-\alpha)r_2)^2}
\end{align*}
One can then check directly that $Q_1$ and $Q_2$ are non-positive for $\alpha$ in the range $[1,2]$.
We can also see from these expressions that the pinching estimate fails for $\alpha>2$ if $r_2/r_1$ is
 large, since the coefficient of $r_2^5$ in the numerator of $Q_2$ is then positive.  One can also
 check that $Q_1$ and $Q_2$ are non-positive for $\alpha=1/2$, and hence by the above argument
 for $\alpha\in[1/2,2]$, but not for $0<\alpha<\frac12$.  The resulting pinching estimate for $\alpha<1$
 is not useful for surfaces in Euclidean space, but is a strong pinching estimate for convex spacelike
 co-compact hypersurfaces in Minkowski space $\R^{2,1}$.  The flow for $\alpha=1/2$ is particularly interesting
  because it is the affine invariant flow
studied for Euclidean convex hypersurfaces in \cite{AndrewsAffine}.
\end{remark}

\begin{remark}
The solution of the `reaction' ODE system in the form $F(\rr_2-\rr_1)=c$ provides a natural choice of pinching quantity to consider for other flows.  The conditions required on $F$ to make the gradient terms favourable to preserve this condition are not simple to understand in any generality, but can be checked in concrete examples:  We announce here some results of this investigation, details of which will be provided in a separate paper:   Applying this recipe for flows by powers $H^\alpha$ of the mean curvature we find the quantity $\sup_{M_t}H^\alpha|\kappa_2-\kappa_1|/K$ is non-increasing under the flow provided $1\leq \alpha\leq\alpha_*$ where $\alpha_*$ is approximately $5.16$.  These flows were considered previously in \cite{Schn}, where pinching estimates were found for $\alpha=2,3,4$, and in \cite{Schu} where pinching quantities were found for $1\leq\alpha\leq 5$.  Similarly we can consider flow by powers of $|A|=\sqrt{\kappa_1^2+\kappa_2^2}$:  Here our pinching estimate $|A|^\alpha|\kappa_2-\kappa_1|/K\leq C$ applies for $\alpha\geq 1$ up to approximately $8.15$ (the values $\alpha=2$ and $\alpha=4$ were considered in \cite{Schn}).   Finally, in the case where $F=\kappa_1^\alpha+\kappa_2^\alpha$, the pinching estimate $F|\kappa_2-\kappa_1|/K\leq C$ holds for any $\alpha>1$.  This is the first example where flows of arbitrarily high degree of homogeneity have been shown to flow convex surfaces to spheres.
\end{remark}



\begin{bibdiv}
\begin{biblist}

\bib{AndrewsContractCalc}{article}{
   author={Andrews, Ben},
   title={Contraction of convex hypersurfaces in Euclidean space},
   journal={Calc. Var. Partial Differential Equations},
   volume={2},
   date={1994},
   number={2},
   pages={151--171},
  }

\bib{AndrewsHarnack}{article}{
   author={Andrews, Ben},
   title={Harnack inequalities for evolving hypersurfaces},
   journal={Math. Z.},
   volume={217},
   date={1994},
   number={2},
   pages={179--197},
   }

\bib{AndrewsAffine}{article}{
   author={Andrews, Ben},
   title={Contraction of convex hypersurfaces by their affine normal},
   journal={J. Differential Geom.},
   volume={43},
   date={1996},
   number={2},
   pages={207--230},
   }

\bib{AndrewsStones}{article}{
   author={Andrews, Ben},
   title={Gauss curvature flow: the fate of the rolling stones},
   journal={Invent. Math.},
   volume={138},
   date={1999},
   number={1},
   pages={151--161},
   }

\bib{AndrewsMHGC}{article}{
   author={Andrews, Ben},
   title={Motion of hypersurfaces by Gauss curvature},
   journal={Pacific J. Math.},
   volume={195},
   date={2000},
   number={1},
   pages={1--34},
  }

\bib{AndrewsICM}{article}{
   author={Andrews, B.},
   title={Positively curved surfaces in the three-sphere},
   conference={
      title={},
      address={Beijing},
      date={2002},
   },
   book={
      publisher={Higher Ed. Press},
      place={Beijing},
   },
   date={2002},
   pages={221--230},
   }

\bib{AndrewsPinching}{article}{
   author={Andrews, Ben},
   title={Pinching estimates and motion of hypersurfaces by curvature
   functions},
   journal={J. Reine Angew. Math.},
   volume={608},
   date={2007},
   pages={17--33},
  }

\bib{AndrewsNonconcave}{article}{
   author={Andrews, Ben},
   title={Moving surfaces by non-concave curvature functions},
   journal={Calc. Var. Partial Differential Equations},
   volume={39},
   date={2010},
   number={3-4},
   pages={649--657},
   }

\bib{AndrewsMcCoy}{article}{
    author={Andrews, Ben},
    author={McCoy, James A.},
    title={Convex Hypersurfaces with pinched
curvatures and flow of convex hypersurfaces by high powers of
curvature},
    journal={Trans. Amer. Math. Soc.},
    status={to appear},
    eprint={arXiv:0910.0376v1 [math.DG]},
    }

\bib{Chow1}{article}{
   author={Chow, Bennett},
   title={Deforming convex hypersurfaces by the $n$th root of the Gaussian
   curvature},
   journal={J. Differential Geom.},
   volume={22},
   date={1985},
   number={1},
   pages={117--138},
   }

\bib{Firey}{article}{
   author={Firey, William J.},
   title={Shapes of worn stones},
   journal={Mathematika},
   volume={21},
   date={1974},
   pages={1--11},
  }

\bib{Ham3D}{article}{
   author={Hamilton, Richard S.},
   title={Three-manifolds with positive Ricci curvature},
   journal={J. Differential Geom.},
   volume={17},
   date={1982},
   number={2},
   pages={255--306},
  }

\bib{Schn}{article}{
   author={Schn{\"u}rer, Oliver C.},
   title={Surfaces contracting with speed $\vert A\vert ^2$},
   journal={J. Differential Geom.},
   volume={71},
   date={2005},
   number={3},
   pages={347--363},
  }

  \bib{Schu}{article}{
   author={Schulze, Felix},
   title={Convexity estimates for flows by powers of the mean curvature},
   journal={Ann. Sc. Norm. Super. Pisa Cl. Sci. (5)},
   volume={5},
   date={2006},
   number={2},
   pages={261--277},
  }

\bib{TsoPoint}{article}{
   author={Tso, Kaising},
   title={Deforming a hypersurface by its Gauss-Kronecker curvature},
   journal={Comm. Pure Appl. Math.},
   volume={38},
   date={1985},
   number={6},
   pages={867--882},
  }

\end{biblist}\end{bibdiv}
\end{document}